\documentclass[12pt]{article} 

\usepackage{amsmath} 
\usepackage{amssymb} 
\usepackage{subfigure}

\usepackage{graphicx} 

\usepackage{cite}

\usepackage{color} 
\usepackage{float} 
\graphicspath{ {images/} }

\usepackage{setspace}
\usepackage{cite} 
\newtheorem{theorem}{Theorem}

\newtheorem{conjecture}{Conjecture}
\newtheorem{corollary}{Corollary}

\newtheorem{lemma}{Lemma}

\newtheorem{proposition}{Proposition}
\newtheorem{remark}{Remark}

\numberwithin{equation}{section}
\input{tcilatex}

\usepackage[dvipsnames]{xcolor}

\definecolor{tab:blue}{HTML}{1f77b4}
\definecolor{tab:orange}{HTML}{ff7f0e}
\definecolor{tab:green}{HTML}{2ca02c}
\definecolor{tab:red}{HTML}{d62728}
\definecolor{tab:purple}{HTML}{9467bd}
\definecolor{tab:brown}{HTML}{8c564b}
\definecolor{tab:pink}{HTML}{e377c2}
\definecolor{tab:gray}{HTML}{7f7f7f}
\definecolor{tab:olive}{HTML}{bcbd22}
\definecolor{tab:cyan}{HTML}{17becf}
\definecolor{RubineRed}{HTML}{ED017D}
\definecolor{Brown}{HTML}{792500}

\begin{document}

\begin{center}
{\Large
\textbf{Random walk in a non-homogeneous random environment on some random trees and the non-negative integers }
}
\\
Tamer Oraby%
\footnote{School of Mathematical and Statistical Sciences, The University of Texas Rio Grande Valley, Edinburg, Texas 78539, United States of America. \texttt{tamer.oraby@utrgv.edu}, Orcid: 0000-0002-8176-1324} \& András Telcs%
\footnote{HUN-REN Wigner RCP Budapest, University Pannonia, Veszprém, Hungary. \texttt{telcs.szit.bme@gmail.com}}
\end{center}
Dedicated to the soul of Dr. Mokhtar Konsowa.
\begin{abstract}
In this paper, we consider a random walk in a non-homogeneous random environment on two types of random trees and non-negative integers. The random trees here have progeny distributions that depend on the parent's generation. We provide a sharp threshold for the type of the random walk in a random environment on the random trees and the non-negative integers, up to the critical case. Then, we find the speed of the random walk on the non-negative integers and on random spherically symmetric trees. We use the method based on the relationship between the mean exit time from a vertex and speed. Meanwhile, we discuss several special cases that could be derived from those results. We also include many limiting formulas for sequences and series in the appendix, which are helpful for this study and others.
\end{abstract}

\section{Introduction\label{S1}}

We consider trees $\textbf{T}\equiv\textbf{(V,E)}$ without loops so that the distance between two vertices in $\textbf{V}$ is defined by the length (number of edges in $\textbf{E}$) of the shortest path that joins them. Let level $S_n$ be the set of vertices at a distance (or height) $n$ from the root $r$, which happens to be the only vertex in $S_0$. A tree is called a spherically symmetric tree ($\textbf{SST}$), if each vertex in level $S_n$ has the same number of children $d^+_n$ (also called the out-degree of a vertex). It is called a random spherically symmetric tree ($\textbf{RSST}$) if $d^+_n$'s are random variables with probability distribution $P(d^+_n=k)=p_{n,k}$ for $k=1,\ldots,D$; that is, the tree is locally bounded and grows without extinction. Note that $d^+_n$ is the same for vertices in generation $n$. A $k$-regular tree is a deterministic tree with $d^+_n=k-1$ for all $n$, and the degree of the root $r$ is equal to its out-degree $d^+_0=k$. A $k$-ary tree is a deterministic tree with $d^+_n=k$ for all $n$, and the degree of the root $r$ equals its out-degree, $d^+_0=k$. The line of nonnegative integers $\mathbb{Z}^+=\{0,1,2,\ldots\}$ is a unary tree with $r\equiv 0$. The line of integers $\mathbb{Z}$ is a 2-regular tree with $r\equiv 0$.

A random tree that is not spherically symmetric but equipped with the same out-degree probability distribution $P(d^+_{n,\ell}=k)=p_{n,k}$ for $k=1,\ldots,D$ and $\ell=1,\ldots, Z_n$; is the genealogical tree of the branching process in a varying environment ($\textbf{BPVET}$). Here, $Z_n$ is the size of the nth generation; see below. Note also that $d^+_{n,\ell}$ are i.i.d. random variables for fixed $n$. The Galton-Watson tree ($\textbf{GWT}$) is a special case of $\textbf{BPVET}$ when $p_{n,k} \equiv p_k$ for all $n$. See \cite{konsowa1991type,lyons1990random,lyons1992random,konsowa2003fractal}.

Let $(\Omega_d,\mathcal{F}_d,\mathbb{P})$ be the probability space to which the random out-degrees are defined. Let $\mathbb{E}_{\mathbb{P}}$ and $\mathbb{V}_{\mathbb{P}}$ be the expected value and variance with respect to $\mathbb{P}$.

In any tree, the ${n}^{th}$ canonical cutset of the tree, $\mathcal{E}_n$, is the set of all edges that connect the vertices of $S_{n-1}$ to the vertices of $S_n$. Let $Z_n:=|S_n|=|\mathcal{E}_n|$ denote the number of vertices in level $S_n$ or the number of edges in the canonical cutset. For $\textbf{RSST}$, $Z_n=\prod_{k=0}^{n-1}d^+_k$; and for $\textbf{BPVET}$, $Z_n=\sum_{\ell=1}^{Z_{n-1}} d^+_{n-1,\ell}$ where $d^+_{n-1,\ell}$ are identically distributed as $d^+_{n-1}$ for all $\ell$. 

It follows from \cite{konsowa2003fractal} that, for $\textbf{RSST}$, $Z_n \asymp n^\gamma$ $\mathbb{P}$-a.s. where $\gamma=\lim_n n \mathbb{E}_{\mathbb{P}}(\log d^+_n)$;\footnote{In this paper, $\lim_{n\rightarrow\infty}$ and $\sum_{n=1}^\infty$ will be abbreviated to $\lim_n$ and $\sum_n$ respectively, for brevity. Limits will also be represented by the arrow "$\to$" when the index approaching $\infty$ is clear. With $a_n \asymp b_n$, we mean $\lim_n\dfrac{\log(a_n)}{\log(b_n)}=1$. Almost surely is abbreviated a.s.} and for $\textbf{BPVET}$, $Z_n \asymp n^{\gamma_v}$ $\mathbb{P}$-a.s. where $\gamma_v=\lim_n n\log(\mathbb{E}_{\mathbb{P}}(d^+_n))$. The growth dimensions $\gamma$ and $\gamma_v$ are just the fractal dimensions $d_f$ of each tree less one; see \cite{T,konsowa2003fractal,konsowa2003dimensions} for a definition and more details. If $\textbf{RSST}$ or $\textbf{BPVET}$ are growing exponentially (that is, the growth measure of the tree $gr(T):=\lim_n Z_n/Z_{n-1}=: g > 1$ with a constant $g$), then $\gamma$ and $\gamma_v$ will be infinite. (We do not want to cause confusion when we use $a_n \asymp n^{\infty}$, but it actually means $\lim_n\dfrac{\log(a_n)}{\log(n)}=\infty$.) All the results below include the case where $\gamma$ and $\gamma_v$ are infinite. 

Let $(\Omega_w,\mathcal{F}_w,\mathbb{Q})$ be another probability space to which the weights (introduced below) are defined. Let $\mathbb{E}_{\mathbb{Q}}$ and $\mathbb{V}_{\mathbb{Q}}$ be the expected value and variance with respect to $\mathbb{Q}$. 

Each edge $e \in \mathcal{E}_n$ of the tree is endowed with different copies of a random positive weight $w_n$. Assume $w_0=1$. The random ratios $r_n:=w_n/w_{n-1}$ are assumed to be independent for all $n$. Let $\alpha:=\lim_n n\mathbb{E}_{\mathbb{Q}}(\log r_n)$ exist, but possibly infinite, positive, or negative. All the results shown in the sequel require that

\noindent \textbf{Condition A.}
$$\sum_{n=2}^\infty \frac{1}{(\log(n))^2} \; \mathbb{V}_{\mathbb{Q}}(\log(r_n))<\infty$$

An example of such an environment is a multiplicative cascade $w_n=\prod_{i=1}^n r_i$ defined by the independent random variables $r_n$ such that $r_n$ is equal to $\ell_n>0$ with probability $q_n$ and $m_n>0$ otherwise. Then $\mathbb{V}_{\mathbb{Q}}(\log(r_n))=(\log \ell_n -\log m_n)^2 \,q_n\,(1-q_n)$. Condition A is true if $\sum_{n=2}^\infty \frac{(\log \ell_n -\log m_n)^2\,q_n\,(1-q_n)}{(\log(n))^2}<\infty$. For example, $q_n=\dfrac{c}{n(\log n)^{1+\epsilon}}$ with constants $c,\epsilon>0$, and $\ell_n=c_1 n^a$ and $m_n=c_2 n^b$ with constants $a$ and $b$ possibly zero and constants $c_1$ and $c_2 \geq 1$, have Condition A satisfied.

For the tree-based statements in Theorems \ref{theorem1} and \ref{theorem2}, in which different edges within the same cutset are assigned independent copies of $w_k$, Condition A needs to be reinforced so that the shared growth exponent governs the entire expanding collection of level weights, rather than just one copy. Define $\xi_n:=\log r_n-\mathbb{E}_{\mathbb{Q}}(\log r_n)$ as the centered log-increments and set $m_n:=\mathbb{E}_{\mathbb{Q}}(\log w_n)=\sum_{i=1}^n\mathbb{E}_{\mathbb{Q}}(\log r_i)$; then, by Lemma \ref{limitlawsb3} (a), Condition A implies that $m_n/\log n\to\alpha$.

\noindent \textbf{Condition A$'$.} The increments $\{\xi_n\}$ are independent and sub-Gaussian; that is, there is a nonnegative sequence $\{\tau_n^2\}$ such that $\sum_{n=2}^\infty\frac{\tau_n^2}{\log n}<\infty$ and
$$\mathbb{E}_{\mathbb{Q}}\,e^{\lambda\xi_n}\leq e^{\lambda^2\tau_n^2/2},$$
for any $\lambda\in\mathbb{R}$ for all $n\geq 1$.

\begin{remark}
Condition A$'$ entails Condition A because $\mathbb{V}_{\mathbb{Q}}(\log r_n)\le \tau_n^2$ and $\sum_n \tau_n^2/(\log n)^2 \le \sum_n \tau_n^2/\log n < \infty$. In addition, writing $V_n:=\sum_{i=1}^n \tau_i^2$, Kronecker’s lemma applied to $\sum_n \tau_n^2/\log n<\infty$ yields
$V_n=o(\log n)$. If the environment is bounded, namely $|\xi_n|\le b_n$ a.s., then by Hoeffding’s lemma it is sub-Gaussian with parameter $\tau_n^2=b_n^2$, so Condition A$'$ becomes $\sum_n b_n^2/\log n<\infty$. For the multiplicative cascade described above with $q_n=c/(n(\log n)^{1+\epsilon})$, $\ell_n=c_1 n^a$, and $m_n=c_2 n^b$, one has $\mathbb{V}_{\mathbb{Q}}(\log r_n)\asymp (\log n)^{1-\epsilon}/n$, and therefore $\sum_{i\le k}\mathbb{V}_{\mathbb{Q}}(\log r_i)\asymp (\log k)^{2-\epsilon}$. When $a\ne b$ and $0<\epsilon<1$, this dominates $\log k$, so on a tree the cascade meets Condition A but \emph{fails} to satisfy
Condition A$'$. By contrast, on $\mathbb{Z}^+$, where only one copy at each level is needed, Condition A is sufficient in all arguments below.
\end{remark}

A random spherically symmetric tree equipped with a random environment $w$ will be denoted by $\textbf{RSST}w$. Similarly, a branching process on a varying-environment tree equipped with a random environment $w$ will be denoted by $\textbf{BPVET}w$.  Let $(\Omega_d\times\Omega_w,\mathcal{F},\mathbb{P}\times\mathbb{Q})$ be the product probability space of $(\Omega_d,\mathcal{F}_d,\mathbb{P})$ and $(\Omega_w,\mathcal{F}_w,\mathbb{Q})$. A statement that is $\mathbb{P}\times\mathbb{Q}$-a.s. follow from $\mathbb{P}$-a.s. and $\mathbb{Q}$-a.s. of two independent events.    

A simple random walk (RW) $\{X_k:k\geq 0\}$ starting at the root (that is, $X_0=r$) moves from one vertex to another over the edge connecting them (taking values in $\mathbf{V}$) with a probability proportional to the weight endowed to that edge. If those weights are random, the process is called a random walk in a random environment (RWRE). Let $\mathbb{RW}^r$ be the quenched law of the random walk starting at the root $r$ defined on $(\mathbf{V}^{\mathbb{N}^+},\mathcal{G}^r_R)$. Let $E_r(\cdot|w)$ and $V_r(\cdot|w)$ be the (quenched) expected value and variance with respect to $\mathbb{RW}^r$. We define the measure $\mathbb{S}_0$ on $(\mathbf{V}^{\mathbb{N}^+}\times\Omega_d\times\Omega_w,\mathcal{G}^r_R\times\mathcal{F})$ by $$\mathbb{S}_0(G\times F)=\int_F \mathbb{RW}^r(G) \; \mathbb{P}\times\mathbb{Q}(df)$$
The type of RWRE is quenched-recurrent if $\mathbb{RW}^r(\{X_n=r \, i.o. \})=1$ and transient if $\mathbb{RW}^r(\{X_n=r \, i.o. \})=0$.

The annealed measure is defined by $\mathbb{S}(\cdot):=\mathbb{S}_0(\cdot\times \Omega_d\times\Omega_w)$. Let $\mathbb{E}_{\mathbb{S}}$ and $\mathbb{V}_{\mathbb{S}}$ be the expected value and variance with respect to $\mathbb{S}$. The stochastic process with respect to $\mathbb{S}$ is not a Markov chain.

The speed of a random walk on a graph is defined to be the rate of
moving away from its starting vertex, which is the root $r$ in trees. That is, with $|X_{n}|$ measuring the distance of the random walk from the root $r$ at time $n$, the speed is defined by $v:=\lim_{n}
\frac{|X_{n}|}{n}$ provided that the limit exists. 
Now, let $T_{n}=\min\{k \geq 0:X_k\in S_n\}$ be the first hitting/passage time to the level $S_n$ of the tree by the random walk. 

The novelty of this paper is most clearly understood when placed alongside two lines of prior work. First, for Galton–Watson and more general trees, the nature and speed of random walks are controlled, following Lyons \cite{lyons1990random,lyons1992random}, by the branching number. Also, for a $\lambda$-biased walk, the critical bias is $\mathrm{br}(\mathbf{T})$. While this provides a sharp criterion for exponentially growing trees, it does not address the setting we consider, where the trees have polynomial growth, $Z_n\asymp n^{\gamma}$ (resp.\ $n^{\gamma_v}$) with $\gamma,\gamma_v<\infty$. In this case, the branching number is always one, and the classification instead depends on finer dimensional information, captured by the growth exponent \cite{T,konsowa2003fractal,konsowa2003dimensions}. Second, the Konsowa line \cite{konsowa1991type,K,K-F} characterizes the type and speed of walks on random spherically symmetric trees precisely in this polynomial-growth regime, but it treats only the geometry of the tree itself, assuming unit or deterministic weights.

Our central contribution is to connect these two sources of randomness and to identify the associated single sharp threshold. We allow (i) a \emph{non-homogeneous} offspring distribution that can change from one generation to the next ($\textbf{RSST}$ or $\textbf{BPVET}$) and (ii) an independent random edge environment $w_n=\prod_{i\le n}r_i$ with logarithmic drift $\alpha=\lim_n n\,\mathbb{E}_{\mathbb{Q}}(\log r_n)$. Theorems \ref{theorem1} and \ref{theorem2} establish a P\'olya-type dichotomy in which the tree geometry and the environmental drift appear symmetrically. Specifically, the two exponents combine additively: the walk is recurrent if $\gamma+\alpha<1$ (resp.\ $\gamma_v+\alpha<1$) and becomes transient once this sum exceeds one. Setting $w_n\equiv1$ yields the purely geometric results of \cite{konsowa1991type,K-F}, while $\gamma=0$ corresponds to the half-line $\mathbb{Z}^+$. For the speed, we study $\mathbb{Z}^+$ and $\textbf{RSST}w$ in an \emph{independent but not identically distributed} environment. Solomon's law and the Kesten-Kozlov-Spitzer exponent \cite{S,kesten1975limit} were originally obtained for i.i.d.\ environments. Moreover, Theorem \ref{theoremZ1} and Theorem \ref{theoremZ2} derive the speed $v=(1-s)/(1+s)$ using a mean-exit-time method through the electric-network framework, where $s=s_w$ on the line and $s=s_T s_w$ on the tree. Corollary \ref{cor1} retrieves the tree-speed statement from \cite{K}, and for the $2$-regular tree we recover Solomon's theorem on $\mathbb{Z}$.

\section{Main Results}

\begin{theorem}\label{theorem1}
The type of RWRE on $\textbf{RSST}w$, under Condition A$'$, which starts at $r$, with $\alpha>0$, is recurrent almost surely if $\gamma+\alpha<1$, and transient almost surely if $\gamma+\alpha>1$.
\end{theorem}

\begin{theorem}\label{theorem2}
The type of RWRE on $\textbf{BPVET}w$, under Condition A$'$, which starts at $r$ is recurrent almost surely if $\gamma_v+\alpha<1$, and transient almost surely if $\gamma_v+\alpha>1$.
\end{theorem}

\begin{remark}
The parameters $\alpha+\gamma$ and $\alpha+\gamma_v$ can be infinite, positive, or negative.
\end{remark}

\begin{remark}
Condition A$'$ is only invoked when moving from a single edge to an expanding cutset on trees. On $\mathbb{Z}^+$ (where $Z_k\equiv1$), and more generally when the weights are uniform within each level ($w_{k,\ell}\equiv w_k$, as in Remark \ref{remSSRT}), Condition A by itself is enough. By contrast, a rare-but-large jump setting—such as the cascade described in Section \ref{S1}—meets Condition A but fails Condition A$'$. For these types of environments, the sharp threshold for $\textbf{RSST}w$ is asserted only assuming Condition A$'$, though on $\mathbb{Z}^+$ the result still holds under Condition A.
\end{remark}

\begin{remark}
\begin{enumerate}
\item In the simple deterministic case (canonical case) when $w_n=n^\alpha$ for $n \geq 1$, it follows from the theorem that the \emph{biased} RW on $\textbf{RSST}w$ (or $\textbf{BPVET}w$) that starts at $r$ is recurrent a.s. if $\gamma+\alpha<1$ ($\gamma_v+\alpha<1$) and is transient a.s. if $\gamma+\alpha>1$ ($\gamma_v+\alpha>1$). If $w_n=1$ (or $\alpha=0$), then the type problem boils down to the type problem of simple random walk on $\textbf{RSST}$ ($\textbf{BPVET}$); see also \cite{konsowa1991type,K-F}.  

\item It follows from Theorem \ref{theorem1} that RWRE on the line of non-negative integers $\mathbb{Z}^+$ (for which $\gamma=0$ and $Z_k\equiv1$, so that Condition A alone suffices) with those random weights starting at $0$ is recurrent a.s. if $\alpha=\lim_n n\mathbb{E}_{\mathbb{Q}}(\log r_n)<1$ and is transient a.s. if $\alpha>1$; the case $\alpha\le0$ is elementary, since then $R_n=\sum_k 1/w_k$ with $\lim_n\log w_n/\log n=\alpha\le0$ diverges, giving recurrence. A similar conclusion follows for the canonical case $w_n=n^\alpha$. 

\end{enumerate}

\end{remark}

We will also give the following result for the speed of RWRE on $\mathbb{Z}^+$. In \cite{kesten1975limit,tavare2004lectures,sznitman2004topics}, the speed was studied when $\{r_n: n \geq 1\}$ is a set of identically distributed random variables. Considering a 2-regular tree ($Z_n=2$ for all $n$) our theorem can retrieve Solomon's result for an RWRE on $\mathbb{Z}$ starting at the origin. See also \cite{peterson2013lecture,zeitouni2009random} for a review of RWRE and the references therein.

\noindent \textbf{Condition B.} The relative weights $r_n\geq c_n$ for all $n\geq 1$, almost surely, where $\{c_n\}$ is a constant sequence such that $\lim_n c_n=c>1$.

\begin{theorem}\label{theoremZ1}
Assume Condition B holds. Then the RWRE on $\mathbb{Z}^+$ started at $0$ has speed
$$v_L:=\dfrac{1-s_w}{1+s_w}\;\; \mathbb{S}\text{-a.s.}$$
provided that $s_w=\lim_n \mathbb{E}_{\mathbb{Q}}\left(\dfrac{1}{r_n}\right)<1$. If Condition B fails, then the speed is $v_L=0$, and consequently $s_w>1$.
\end{theorem}
The quantities $\alpha$ and $s_w$ capture distinct aspects of the walk. The sign of $\mathbb{E}_{\mathbb{Q}}(\log\rho)$, where $\rho:=1/r$, determines the walk’s type. Meanwhile, $s_w$ controls whether the walk is ballistic within the transient phase. When $\{r_n\}$ are i.i.d., the model coincides with Solomon’s walk with ratios $\rho_k=1/r_k$. Define $g(\theta):=\mathbb{E}_{\mathbb{Q}}(\rho^\theta)$ which is a convex function, with $g(0)=1$, $g'(0)=\mathbb{E}_{\mathbb{Q}}(\log\rho)$, and $g(1)=s_w$. By Solomon \cite{S}, on $\mathbb{Z}^+$ the walk is recurrent exactly when $\mathbb{E}_{\mathbb{Q}}(\log\rho)\ge0$, and it is transient to $+\infty$ exactly when $\mathbb{E}_{\mathbb{Q}}(\log\rho)<0$; thus the type depends only on the sign of $g'(0)$. In the transient regime, let $\kappa>0$ be the solution of $g(\kappa)=1$; see Kesten–Kozlov–Spitzer \cite{kesten1975limit}. Convexity implies $\operatorname{sign}(\kappa-1)=\operatorname{sign}(1-s_w)$, so 
\begin{enumerate}
    \item $s_w<1$ if and only if $\kappa>1$ indicating the ballistic case with $v_L=(1-s_w)/(1+s_w)>0$,
    \item $s_w=1$ if and only if $\kappa=1$ indicating the critical zero-speed case, and
    \item $s_w>1$ if and only if $0<\kappa<1$ indicating the transient but sub-ballistic case, $X_n\asymp n^\kappa$ and $v_L=0$.
\end{enumerate}
Note that $s_w<1$ implies transience with $v_L>0$ by Jensen’s inequality, since $$\mathbb{E}_{\mathbb{Q}}(\log\rho)\le \log s_w<0.$$
The converse, however, does not hold. In particular, having $s_w>1$ does not force recurrence. As an illustration, let $r\in\{3,\frac13\}$ with $\mathbb{Q}(r=3)=0.6$. Then $\mathbb{E}_{\mathbb{Q}}(\log\rho)=-0.2\log 3<0$, so the walk is transient, yet $s_w=1.4>1$. Furthermore, solving $g(\kappa)=0.6\cdot 3^{-\kappa}+0.4\cdot 3^{\kappa}=1$ yields $\kappa=\log_3(3/2)\approx 0.369$, so the walk drifts to $+\infty$ with $X_n\asymp n^{0.369}$ but still has zero limiting velocity. See Conjecture \ref{conj_1}. Thus, Condition B is best interpreted as a sufficient regularity condition rather than the fundamental dichotomy. It imposes $s_w\leq 1/c<1$, which then leads to the ballistic statement of Theorem \ref{theoremZ1}, and its failure need not imply $s_w>1$. The actual dichotomy is $s_w<1$ with speed $(1-s_w)/(1+s_w)$ versus $s_w\ge 1$, which is of zero speed, with the borderline case $s_w=1$ appearing in the cascade example below.

The value of $s_w$ might equal one, resulting in a zero-speed RWRE even if it is transient. For example, if the weights are given by a multiplicative cascade $w_n=\prod_{i=1}^n r_i$ defined by the independent random variables $r_n$ such that $r_n$ is equal to $\ell>0$ with probability $\frac{1}{n}$ and equal to $1$ otherwise. Then Condition A is satisfied and $\alpha=\lim_n n \mathbb{E}_{\mathbb{Q}}(\log r_n)=\log \ell$. Therefore, RWRE is transient when $\ell>\exp(1)$ and is recurrent otherwise for that environment. Meanwhile, it has $s_w=1$. Based on that point and the polynomial growth of some branching processes frequently used to study hitting time, we make the following conjecture.

\begin{conjecture}\label{conj_1}
When $s_w=1$ under conditions A and B, then $\alpha$ determines a phase transition in the logarithmic scale speed, $\lim_n \dfrac{\log |X_n|}{\log n}$, of an RWRE in $\mathbb{Z}^+$; this occurs primarily at the value of $\alpha=1$.
\end{conjecture}

The following corollary generalizes the results in \cite{K}.
\begin{corollary}\label{cor1}
Suppose the out-degrees satisfy $d_n^+\ge c_n$ a.s.\ for all $n\geq 1$, where $\{c_n\}$ is a constant sequence with $\lim_n c_n=c>1$. Then the speed of a simple symmetric random walk on $\textbf{RSST}$, which starts at $r$, is 
$$v_T:=\dfrac{1-s_T}{1+s_T} \;\; \mathbb{S}-\text{a.s.}$$
if $s_T=\lim_n \mathbb{E}_{\mathbb{P}}(\dfrac{1}{d_n^+})<1$. The speed $v_T=0$ when $s_T\ge1$.
\end{corollary}
\begin{proof}
The distance from the origin in a simple symmetric random walk on $\textbf{RSST}$ can be viewed as an RWRE on $\mathbb{Z}^+$ with transition parameters $r_n=d_n^+$ (see \cite{K}). Under this identification, the assumption coincides precisely with Condition B for the resulting effective environment, and the statement follows by applying Theorem \ref{theoremZ1} with $r_n=d_n^+$.
\end{proof}

For $\textbf{RSST}w$, both the type and the speed collapse to those of the associated line, since the tree is spherically symmetric. From any level-$k$ vertex ($k\ge1$), the walk either moves down to one of its $d_k^+$ children along an edge in $\mathcal{E}_{k+1}$ with conductance $w_{k+1}$, or moves up to its parent along the unique edge in $\mathcal{E}_k$ with conductance $w_k$. Moreover, all vertices at level $k$ are indistinguishable. Therefore the distance-to-root process $\{|X_n|\}_{n\geq 1}$ forms a birth-death chain on $\mathbb{Z}^+$ with $q_k/p_k=w_k/(d_k^+ w_{k+1})=1/(d_k^+ r_{k+1})$, i.e., it is the RWRE from Theorem \ref{theoremZ1} with effective inverse ratios $s_k:=\frac{1}{d_k^+\,r_{k+1}}$ for $k\geq 1$. Since $\{d_k^+\}$ under $\mathbb{P}$ and $\{r_k\}$ under $\mathbb{Q}$ are independent across $k$ and also independent of each other, the variables $\{s_k\}$ are independent, and $\bar s_k:=\mathbb{E}_{\mathbb{P}\times\mathbb{Q}}(s_k)=\mathbb{E}_{\mathbb{P}}(1/d_k^+)\,\mathbb{E}_{\mathbb{Q}}(1/r_{k+1})$. Hence
$$s_{Tw}:=\lim_k\bar s_k=\Bigl(\lim_k\mathbb{E}_{\mathbb{P}}(\frac1{d_k^+})\Bigr)
\Bigl(\lim_k\mathbb{E}_{\mathbb{Q}}(\frac1{r_k})\Bigr)=s_T\,s_w.$$

\noindent\textbf{Condition B$'$ (joint).} The effective ratios obey $d_k^+\,r_{k+1}\ge\tilde c_k$ for every $k\ge1$, almost surely, where $\{\tilde c_k\}$ is a constant sequence satisfying $\lim_k\tilde c_k=\tilde c>1$.

Condition B$'$ is less restrictive than Condition B. Since $d_k^+\ge1$, every environment that satisfies Condition B also satisfies B$'$ with $\tilde c=c$. However, B$'$ additionally admits environments where $s_w$ is close to or even exceeds $1$, as long as sufficient branching makes up for it.

\begin{theorem}\label{theoremZ2}
Given Condition B$'$, an RWRE on $\textbf{RSST}w$ that starts at $r$ has speed
$$v_{Tw}:=\frac{1-s_{Tw}}{1+s_{Tw}}=\frac{1-s_T\,s_w}{1+s_T\,s_w}\qquad\mathbb{S}\text{-a.s.},$$
where $s_T=\lim_n\mathbb{E}_{\mathbb{P}}(1/d_n^+)$ and $s_w=\lim_n\mathbb{E}_{\mathbb{Q}}(1/r_n)$.
The speed $v_{Tw}=0$ when $s_{Tw}\geq 1$.
\end{theorem}

\begin{remark}
Theorem \ref{theoremZ2} subsumes two previous statements. If we take $r_n\equiv1$, then $s_w=1$ and $s_{Tw}=s_T$, yielding the simple-symmetric speed $v_T=(1-s_T)/(1+s_T)$ from Corollary \ref{cor1}. If instead we set $d_n^+\equiv1$ (i.e., the line), then $s_T=1$ and $s_{Tw}=s_w$, which reproduces Theorem \ref{theoremZ1}. The factorization $s_{Tw}=s_T\,s_w$ indicates that, for a spherically symmetric tree, the contributions of the tree’s geometry and the edge environment influence the speed independently, via their respective harmonic-type means. This, moreover, refines Remark \ref{remSSRT}, in which the same product $s_w s_T$ appeared for level-constant weights.
\end{remark}

\begin{remark}
It may be tempting to describe the speed on $\textbf{BPVET}w$ using the annealed one-step drift
$$v_{BT_w}:=\lim_n\mathbb{E}_{\mathbb{P}\times\mathbb{Q}}
\left(\frac{1-\frac1{r_n d_n^+}}{1+\frac1{r_n d_n^+}}\right),$$
which is analogous to $v_{Tw}$, with the difference that the expectation is applied \emph{after} computing the local drift rather than beforehand. Nevertheless, two points should be separated. First, proving the appropriate inequality that connects these two quantities, and second, the still-unconfirmed assertion that $v_{BT_w}$ coincides with the true a.s. speed.

\emph{The inequality is correct.} The map $f(x)=(1-x)/(1+x)=-1+2/(1+x)$ is convex on $(0,\infty)$, so by Jensen's inequality, with $s_n=1/(r_n d_n^+)$,
$$\mathbb{E}\left(\frac{1-s_n}{1+s_n}\right)\geq\frac{1-\mathbb{E}(s_n)}{1+\mathbb{E}(s_n)},$$ whence $v_{BT_w}\ge v_{Tw}$. Hence, the annealed drift of $\textbf{BPVET}w$ is at least as large as the $\textbf{RSST}w$ speed from Theorem \ref{theoremZ2}, which aligns with the intuition that offspring variability can only help the walk escape.

\emph{Why this does not prove the speed.} In contrast to $\textbf{RSST}w$, the tree $\textbf{BPVET}w$ lacks spherical symmetry. Vertices in the same generation have descendant subtrees of varying sizes, so the distance sequence $\{|X_n|\}$ is \emph{not} Markov. Equivalently, the $\mathbb{S}$-process is not Markov. Thus the birth-death reduction used in Theorem \ref{theoremZ2} cannot be applied. Rayleigh monotonicity, together with the cut/short laws, relates resistances between these trees and hence transports the type (Theorem \ref{theorem2}). But it does not yield any control on the speed. For Galton-Watson trees, the speed of a biased or simple random walk is determined by the environment viewed from the particle and is typically a non-explicit function of the offspring distribution \cite{lyons1990random,aidekon2014speed,zeitouni2009random,peterson2013lecture}. The neat average $v_{BT_w}$ would arise only if the walk mixed within each generation, which is generally obstructed by the size-biasing inherent in the environment-from-the-particle. Accordingly, we expect $v_{BT_w}$ to be a strict heuristic upper bound rather than the actual speed, and we note that identifying the true speed remains an open problem.

\begin{conjecture}\label{conj_2}
Under Condition B$'$, the RWRE on $\textbf{BPVET}w$ has an a.s. speed $v\in[v_{Tw},v_{BT_w}]$, with $v=v_{BT_w}$ in the degenerate case of vanishing offspring fluctuation, i.e., $\mathbb{V}_{\mathbb{P}}(d_n^+)\to0$ as $n\to \infty$, and $v<v_{BT_w}$ in general. A proof would proceed via regeneration times and the stationary environment seen from the particle, as in the Galton-Watson theory \cite{aidekon2014speed}.
\end{conjecture}
\end{remark}

\section{Proofs\label{S2}}
A number of auxiliary lemmas and propositions necessary for the proof of the results are given in the appendix. In the following, we give a proof of Theorem \ref{theorem1}.

\begin{lemma}\label{leveluniform}
Let $\mathbf{T}$ be $\textbf{RSST}w$ or $\textbf{BPVET}w$ with independent edge-weight copies $\{w_{k,\ell}:\ell=1,\dots,Z_k\}$ within each cutset. Suppose $\log Z_k/\log k\to g$ $\mathbb{P}$-a.s.\ with $g<\infty$ with $g=\gamma$ or $g=\gamma_v$, respectively. Let $m_k=\sum_{i=1}^k \mathbb{E}_{\mathbb{Q}}(\log r_i)$. Under Condition A$'$,
$$\lim_k \frac{\max_{1\le\ell\le Z_k}\bigl|\log w_{k,\ell}-m_k\bigr|}{\log k}=0 \qquad \mathbb{P}\times\mathbb{Q}\text{-a.s.},$$
and consequently $\sum_{\ell=1}^{Z_k}w_{k,\ell}\asymp k^{g+\alpha}$, $\mathbb{P}\times\mathbb{Q}$-a.s.
\end{lemma}
\begin{proof}
Take $\epsilon>0$ and define $\Xi_{k,\ell}:=\log w_{k,\ell}-m_k=\sum_{i=1}^k\xi_{i,\ell}$, which is the sum of $k$ independent sub-Gaussian increments with variance proxy $V_k=\sum_{i=1}^k\tau_i^2$. By the Chernoff bound, for every $\ell$ we have $\mathbb{Q}(|\Xi_{k,\ell}|>\epsilon\log k)\le 2\exp\!\big(-\epsilon^2(\log k)^2/(2V_k)\big)$.

Because $\log Z_k/\log k\to g$, there exists a $\mathbb{P}$-a.s.\ event such that for all sufficiently large $k$ one has $Z_k\le k^{g+1}$. Conditioning on $\mathcal{F}_d$ and applying a union bound across the $Z_k$ replicas yields
$$\mathbb{Q}\Bigl(\max_{\ell\le Z_k}|\Xi_{k,\ell}|>\epsilon\log k\,\Big|\,\mathcal{F}_d\Bigr)
\leq 2\,k^{g+1}\exp\Bigl(-\frac{\epsilon^2(\log k)^2}{2V_k}\Bigr).$$
Under Condition A$'$, we have $V_k=o(\log k)$, so for all large $k$ the exponent is larger than $(g+3)\log k$. Consequently, the right-hand side is bounded by $2k^{-2}$, which is summable. It then follows from the Borel-Cantelli lemma that $\max_{\ell\le Z_k}|\Xi_{k,\ell}|\le\epsilon\log k$ eventually, $\mathbb{P}\times\mathbb{Q}$-a.s. Taking the intersection over $\epsilon=1/m$ gives the first assertion. 

Define $\delta_k:=\max_{\ell\leq Z_k}|\Xi_{k,\ell}|=o(\log k)$. Then $Z_k e^{m_k-\delta_k}\le\sum_{\ell\le Z_k}w_{k,\ell}\leq Z_k e^{m_k+\delta_k}$, and therefore $\log(\sum_{\ell\leq Z_k} w_{k,\ell})/\log k\to g+\alpha$, since $\log Z_k/\log k\to g$, $m_k/\log k\to\alpha$, and $\delta_k/\log k\to 0$.
\end{proof}

\begin{remark}
    In case of, $w_{n,\ell}\equiv w_n$ for all $\ell$. Under Condition A, the generalized Kolmogorov law of large numbers, see \cite{robbins1985convergence} or \cite[Corollary 7.4.1]{resnick2013probability}, implies that $$\dfrac{\log(w_n)-E(\log(w_n))}{\log n}=\dfrac{\sum_{k=1}^{n}(\log(r_k)-E(\log(r_k))}{\log n}$$ goes to zero as $n \longrightarrow
\infty$, almost surely. Part (1) of Lemma \ref{limitlawsb3} implies that $m_n/\log n\to\alpha$ and so does $\dfrac{\log(w_n)}{\log n}\to\alpha$. 
\end{remark}

\begin{proof}[Proof of Theorem \ref{theorem1}] 
Fix $n\geq 1$. Consider an electrical network in which every edge $e\in\mathcal{E}_k$ (for $k\leq n$) is assigned a conductance $w_{k,\ell}$, where $\ell=1,\ldots,Z_k$. Inject one unit of current at the root $r$ and remove it at $S_n$. Let $R_n$ denote the resulting effective resistance between $r$ and $S_n$.

Assume that $\gamma+\alpha<1$. If, in the electrical network, we contract (short) the vertices within each level set $S_k$ for all $k\le n$, then the effective resistance in the modified network between $r$ and $S_n$ is
$$
\sum_{k=1}^n \frac{1}{\sum_{\ell=1}^{Z_k} w_{k,\ell}},
$$
and this quantity provides a lower bound for $R_n$ \cite{D-S}. Since $\sum_{\ell=1}^{Z_k} w_{k,\ell}\asymp k^{\gamma+\alpha}$, $\mathbb{P}\times\mathbb{Q}$-a.s. due to Lemma \ref{leveluniform}, then recurrence follows from Lemmas \ref{limitlaws3} and \ref{limitlaws4}. 


Notice that, assuming Condition A$'$, Lemma \ref{leveluniform} with $g=\gamma$ yields a $\mathbb{P}\times\mathbb{Q}$-a.s.\ event $B$ such that for every $\omega\in B$ and each $\epsilon>0$ there exists $n_0(\omega)$ for which
\begin{equation}\label{eq1n0}
n^{\alpha-\epsilon} \leq w_{n,\ell}(\omega)\leq n^{\alpha+\epsilon}
\qquad\text{uniformly for } 1\le\ell\le Z_n,
\end{equation}
holds for all $n>n_0(\omega)$. In other words, $\sum_{\ell=1}^{Z_k} w_{k,\ell}\asymp k^{\gamma+\alpha}$ $\mathbb{P}\times\mathbb{Q}$-a.s.

Assume now that $\gamma+\alpha>1$. Let $T^*$ be the tree obtained by removing all edges of the $\textbf{RSST}w$ from the root $r$ down to level $S_{n_0}$, except for a collection of edges that forms a single path from $r$ to some chosen vertex at that level (see Figure \ref{fig:tree-trimmed}). For simplicity, we take this path to consist of the edges with weights $w_{k,1}$ for $k=1,\ldots,n_0$. Beyond this path, and for each $n>n_0$, assign to every remaining edge in $\mathcal{E}_n$ the weight (conductance) $n^{\alpha-\epsilon}$. By Rayleigh’s Monotonicity Law \cite{D-S}, increasing edge resistances can only increase the total effective resistance; see Equation \eqref{eq1n0}. Hence, using also the Cut law \cite{D-S}, the total effective resistance of $T^*$ is $\mathbb{P}\times\mathbb{Q}$-a.s. larger than that of $\textbf{RSST}w$. Moreover, for every $n>n_0$, the effective resistance $R_n^*$ in $T^*$ between $r$ and level $S_n$ satisfies
\begin{equation}\label{eq2}
R_n^*=\sum_{k=1}^{n_0} \dfrac{1}{w_{k,1}}+\sum_{k=n_0+1}^{n}\dfrac{Z_{n_0}}{Z_k\,k^{\alpha-\epsilon}}.
\end{equation}
By Lemmas \ref{limitlaws3} and \ref{limitlaws4}, the left-hand side converges as $n\to\infty$ whenever $\gamma+\alpha>1$. Finally, one may cut the path from the root to the vertex in $S_{n_0}$ and instead inject the current at the new root (the vertex in $S_{n_0}$), and the conclusion still holds.
\end{proof}

\begin{figure}[H]
    \centering
    \includegraphics[width=0.6\textwidth]{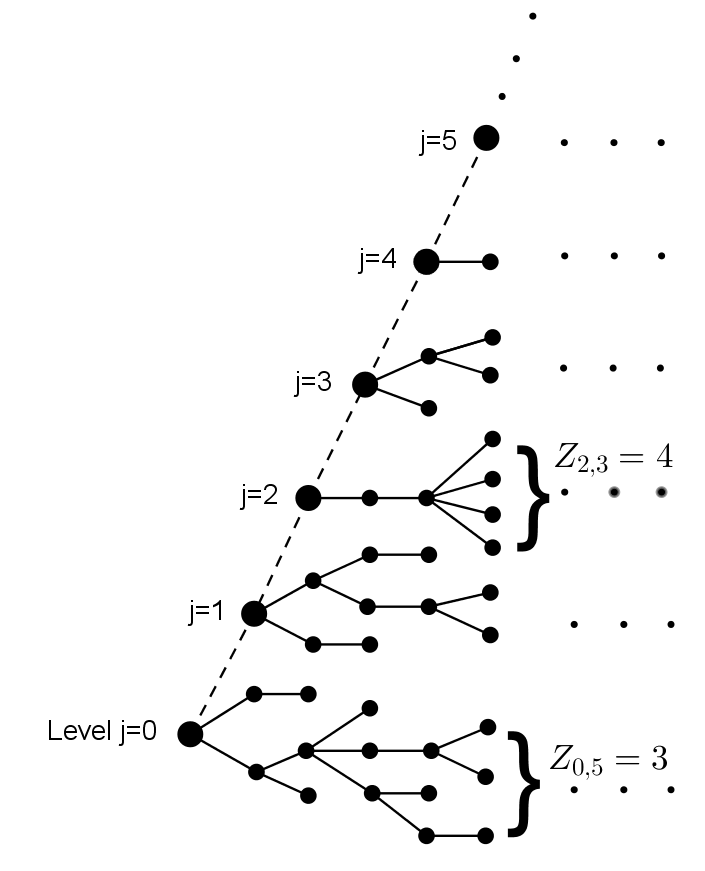}
    \caption{A tree could be trimmed at level $n_0$ so that the subtrees at a distance of $n_0+1$ along the dashed line and above were retained, and the rest of the subtrees to the right of the dashed line at a distance of $n_0$ and below are cut out. For a large $ n$-generation tree, this will increase the total resistance. The new tree at level $n_0+1$ has the same topology as the original tree but with a shifted progeny distribution.}
    \label{fig:tree-trimmed}
\end{figure}

\begin{remark}\label{remSSRT}
In the previous proof, if $w_{k,\ell}=w_k$ for all $\ell=1,\ldots,Z_k$, then the electric current branches out equally at each vertex and so the total effective resistance of the network between $r$ and $S_n$ is given by $R_n:=\sum_{k=1}^n \frac{1}{Z_k\,w_k}$. The limiting sum is almost surely divergent if $\lim_n \frac{\log Z_n \, w_n}{\log n}<1$ $\mathbb{P}\times\mathbb{Q}$-a.s. and is almost surely convergent if $\lim_n \frac{\log Z_n \, w_n}{\log n}>1$ $\mathbb{P}\times\mathbb{Q}$-a.s., and so the results follow in that special case. The speed is then given by $v_{ST}:=\dfrac{1-s_w \, s_T}{1+s_w \, s_T}$, if $s_w \, s_T<1$.
\end{remark}

\begin{proof}[Proof of Theorem \ref{theorem2}] 
The argument proceeds along the same lines as the proof of Theorem \ref{theorem1}. The recurrent regime ($\gamma_v+\alpha<1$) again follows by applying Lemma \ref{leveluniform} with $g=\gamma_v$, exactly as in the proof of Theorem \ref{theorem1}.  

If instead $\gamma_v+\alpha>1$, we repeat the same reductions to obtain the tree T$^*$ (see Figure \ref{fig:tree-trimmed}), obtained by removing the path, whose effective resistance provides an upper bound for that of the $\textbf{BPVET}w$. This modified tree is again a $\textbf{BPVET}$, with out-degree distribution $P(d^{*+}_{n,\ell}=k)=p_{n_0+n,k}$ for $k=1,\ldots,D$, and with cutset size $|\mathcal{E}^*_n|=Z^*_n$, where $Z^*_n=\sum_{\ell=1}^{Z^*_{n-1}} d^{*+}_{n-1,\ell}$. Each edge in $\mathcal{E}^*_n$ is assigned weight (conductance) $(n_0+n)^{\alpha-\epsilon}$. 

Regarding transience, \cite[Theorem 4.5]{lyons1992random} shows that the random walk on the modified tree is a.s. transient as long as
$$
\sum_{n=1}^{\infty}\frac{1}{E(Z^*_n)\,(n_0+n)^{\alpha-\epsilon}}<\infty.
$$
Because the offspring number is bounded ($1\le d^{*+}\le D$) and its mean is uniformly bounded away from zero, the necessary integrability condition is satisfied, and the criterion can therefore be applied to the shifted $\textbf{BPVET}$. In addition,
$E(Z^*_n)=\frac{E(Z_n)}{E(Z_{n_0})}$, and a Kesten–Stigum-type growth estimate for a varying-environment branching process with bounded offspring yields $E(Z_n)\asymp n^{\gamma_v}$ (see \cite{lyons1992random}). It follows that the summand has a logarithmic exponent of $-(\gamma_v+\alpha-\epsilon)$, so the series converges whenever $\gamma_v+\alpha>1$, since $\epsilon>0$ is arbitrary.
\end{proof}

We decompose the proof of Theorem \ref{theoremZ1} into several lemmas. A related form of the next lemma appears in \cite{goldsheid2007simple}; see also \cite{alili1999asymptotic}. While it could be established via mathematical induction, we instead present an electrical network argument for part (a) and give a reference for part (b).

\begin{lemma} Consider a simple random walk on a weighted line $\mathbb{Z}^+$ and starting at 0.
\begin{enumerate}
\item[(a)] The mean time to cross from $k$ to $k+1$, $\mu_k(w)=E(\tau_{k+1}\mid w)$, is given by \begin{equation}
\mu_k(w)=1+2 \sum_{j=1}^{k} \prod_{i=j}^k \dfrac{1}{r_i}.
\end{equation} and the mean hitting time of vertex $n$ is $E_0(T_n|w):=\sum_{k=0}^{n-1} \mu_k(w)$.

\item[(b)] The variance of the hitting time of vertex $n$ is $V_0(T_n|w):=\sum_{k=0}^{n-1} \sigma^2_k(w)$ and
\begin{equation}\label{sigma}
\sigma^2_k(w)=4\sum_{j=0}^{k} (\dfrac{1}{r_j}+\dfrac{1}{r_j^2}) (1+\sum_{\ell=1}^{j-1} \prod_{m=\ell}^{j-1} \dfrac{1}{r_m})^2 \prod_{i=j+1}^k \dfrac{1}{r_i}.
\end{equation}
\end{enumerate}

\end{lemma}
\begin{proof}
Part (a) follows from the well-known relationship, given an environment $w$, $E_0(T_n|w)=\sum_{x\in \mathbf{V}} c_x v_x$, see \cite{D-S}, where $v_x$ is the voltage at the vertex $x$ and $c_x=\sum_{y\sim x} w_{x,y}$ in a network of conductances $w_{x,y}$ connecting vertices $x$ and $y$ and a unit current injected into $0$ and removed from $n$. The voltage $v_x$ at the vertex $x\in \mathbb{Z}^+$ ($x<n$) is given by the effective resistance between $x$ and $n$, that is, $v_x=\sum_{j=x}^{n-1} w^{-1}_{j}$ and $c_x=w_{x-1}+w_{x}$. Since $v_n=0$, then $E_0(T_n|w)=\sum_{j=0}^{n-1}(w_{j-1}+w_{j})(\sum_{k=j}^{n-1} w^{-1}_{k})$ with $w_{-1}=0$. Therefore, $E_0(T_n|w)=\sum_{j=0}^{n-1}\sum_{k=j}^{n-1}(w_{j-1}+w_{j})( w^{-1}_{k})=\sum_{k=0}^{n-1}\sum_{j=0}^{k}(w_{j-1}+w_{j})( w^{-1}_{k})$. Thus, $E_0(T_n|w)=\sum_{k=0}^{n-1}( w^{-1}_{k})\sum_{j=0}^{k}(w_{j-1}+w_{j})=\sum_{k=0}^{n-1}(1+2w^{-1}_{k}\sum_{j=0}^{k-1}w_{j})$. But, $\sum_{j=0}^{k-1}\dfrac{w_{j}}{w_k}=\sum_{j=1}^{k} \prod_{i=j}^k \dfrac{1}{r_i}$. Hence, part (a) is proved.

Part (b) is true; see \cite[Lemma 3]{goldsheid2007simple}. 
\end{proof}

Let $\dfrac{1}{r_n}=s_n$ and $\bar{s}_n=\mathbb{E}_{\mathbb{Q}}(s_n)$ and $\bar{s}=\lim_n \bar{s}_n$. Note that $\mu_{k+1}(w)=1+\dfrac{1+\mu_k(w)}{r_{k+1}}$ or equivalently $\mu_{k+1}(w)=1+(1+\mu_k(w))\,s_{k+1}$, with $\mu_0=1$. If $\{r_n:n\geq1 \}$ are independent then
\begin{equation}
\mathbb{E}_{\mathbb{Q}}(\mu_k(w))=1+2 \sum_{j=1}^{k} \prod_{i=j}^k \bar{s}_i=1+2 \dfrac{\sum_{j=1}^k \prod_{i=1}^{j-1} \bar{s}_i^{-1}}{\prod_{i=1}^k \bar{s}_i^{-1}}.
\end{equation}

\begin{lemma}\label{limitmean1}
If $\{r_n:n\geq1 \}$ are independent, and
$\bar{s}<1$, then $\lim_k \mathbb{E}_{\mathbb{Q}}(\mu_k(w))=\dfrac{1+\bar{s}}{1-\bar{s}}$.
\end{lemma}

\begin{proof}
Since $\lim_j \dfrac{\prod_{i=1}^j \bar{s}_i^{-1}}{\prod_{i=1}^{j-1} \bar{s}_i^{-1}}=\lim_j \bar{s}_j^{-1}=\bar{s}^{-1}>1$ a.s., and ${\sum_{j} \prod_{i=1}^{j-1} \bar{s}_i^{-1}}$ diverges a.s. as $\bar{s}^{-1}>1$ (root test); thus, lemma \ref{limitlaw2x} (b) implies that $\lim_k \dfrac{\sum_{j=1}^k \prod_{i=1}^{j-1} \bar{s}_i^{-1}}{\prod_{i=1}^k \bar{s}_i^{-1}}=\dfrac{1}{\bar{s}^{-1}-1}$ a.s. 

\end{proof}

The following proposition puts the last lemma in perspective. 

\begin{proposition}\label{proposlimit1}
If $\{r_n:n\geq1 \}$ are independent, and
$\bar{s}<1$, then $\lim_n \dfrac{\mathbb{E}_{\mathbb{S}}(T_n)}{n}=\dfrac{1+\bar{s}}{1-\bar{s}}$.
\end{proposition}
\begin{proof}
It follows from lemmas \ref{limitmean1} and \ref{limitlaws2}.
\end{proof}

    



Under Condition B and according to the law of total variance, $\sigma^2_k:=\mathbb{E}_{\mathbb{Q}}(\sigma^2_k(w))+\mathbb{V}_{\mathbb{Q}}(\mu_k(w))$. Thus,
\begin{equation}
\sigma^2_k=4\sum_{j=0}^{k} (E(\dfrac{1}{r_j})+E(\dfrac{1}{r_j^2})) E((1+\sum_{\ell=1}^{j-1} \prod_{m=\ell}^{j-1} \dfrac{1}{r_m})^2) \prod_{i=j+1}^k E(\dfrac{1}{r_i})+4V(\sum_{j=1}^{k} \prod_{i=j}^k \dfrac{1}{r_i})
\end{equation}
and
\begin{equation}
\sigma^2_k\leq 4\sum_{j=0}^{k} (\dfrac{1}{c_j}+\dfrac{1}{c_j^2}) (1+\sum_{\ell=1}^{j-1} \prod_{m=\ell}^{j-1} \dfrac{1}{c_m})^2 \prod_{i=j+1}^k \dfrac{1}{c_i}+4(\sum_{j=1}^{k} \prod_{i=j}^k \dfrac{1}{c_i})^2.
\end{equation}

Let $C_k=\sum_{j=1}^{k} \prod_{i=j}^k \dfrac{1}{c_i}$. But according to the proof of Lemma \ref{limitmean1}, $\lim_k C_k=\dfrac{1}{c-1}$. Also,  
$$\sum_{j=0}^{k} (\dfrac{1}{c_j}+\dfrac{1}{c_j^2}) (1+\sum_{\ell=1}^{j-1} \prod_{m=\ell}^{j-1} \dfrac{1}{c_m})^2 \prod_{i=j+1}^k \dfrac{1}{c_i}=\dfrac{1}{\prod_{i=1}^k c_i}\sum_{j=0}^{k} (1+c_j) C_j^2 \prod_{i=1}^j c_i$$ and the latter is equal to $$(1+c_{k+1}) C_{k+1}^2 c_{k+1}\dfrac{\sum_{j=0}^{k} (1+c_j) C_j^2 \prod_{i=1}^j c_i}{(1+c_{k+1})C_{k+1}^2 \prod_{i=1}^{k+1} c_i}.$$

Let $a_k=(1+c_{k}) C_{k}^2 \prod_{i=1}^{k} c_i$, then $\dfrac{a_{k+1}}{a_k}=\dfrac{(1+c_{k+1})C_{k+1}^2}{(1+c_{k})C_{k}^2} c_{k+1}$.  Thus, $\lim_k \dfrac{a_{k+1}}{a_k}=c>1$, hence lemma \ref{limitlaw2x} (b) implies that $$\lim_k \sigma^2_k \leq 4 (1+c) (\dfrac{1}{c-1})^2 c \dfrac{1}{c-1}+4 \dfrac{1}{c-1}$$ and so the limit is finite.

The hitting time of $n$, $T_n=\sum_{k=0}^{n-1}\tau_{k+1}$ where $\tau_{k+1}$ is the time between hitting $k$ and hitting $k+1$ which are independent for all $k$ given the environment, and $E(\tau_{k+1}\mid w)=\mu_k(w)$ and $V(\tau_{k+1}\mid w)=\sigma^2_k(w)$.

Recall $s_i=1/r_i$, and write $\eta_k:=\sum_{j=1}^{k}\prod_{i=j}^{k}s_i$, so that $\mu_k=1+2\eta_k$ and $\eta_k=s_k(1+\eta_{k-1})$ with $\eta_0=0$. Under Condition B one has $r_i\geq c_i$, hence $s_i\leq 1/c_i$ a.s.; since $c_i\to c>1$, there are $i_0$ and $\rho_0:=\frac{2}{1+c}<1$ with $s_i\leq \rho_0$ for $i\ge i_0$, and $0\leq s_i\leq M:=\sup_j(1/c_j)<\infty$, for all $i$. In particular, $\bar s\leq 1/c<1$ automatically, and
\begin{equation}\label{eqstar}
\eta_k=\sum_{j=1}^{k}\prod_{i=j}^{k}s_i\le M^{\,i_0}\sum_{m=i_0}^\infty\rho_0^{m-i_0}=\frac{M^{\,i_0}}{1-\rho_0}=:C_0<\infty
\quad\text{a.s.}
\end{equation}
for all $k$. Hence, $\mu_k(w)\leq 1+2 C_0$ a.s.

\begin{lemma}\label{newlemma1}
Under Condition B, the hitting time of RWRE on $\mathbb{Z}^+$ satisfies
\begin{equation}
\lim_n \dfrac{T_n}{n}=\dfrac{1+\bar{s}}{1-\bar{s}}:=\dfrac{1}{v_L} \;\; \mathbb{S}\text{-a.s.}
\end{equation}
\end{lemma}
\begin{proof}
We begin by showing the quenched SLLN for $T_n$. Each term appearing in $\sigma^2_k(w)$ is nondecreasing in the variables $s_i$; see equation \eqref{sigma}. Because $s_i\leq 1/c_i$ almost surely, it follows that $\sigma^2_k(w)\leq B_k$, where $B_k$ is obtained from $\sigma^2_k(w)$ by replacing every $s_i$ with $1/c_i$. Using the calculation just before this lemma (taking $a_k=(1+c_k)C_k^2\prod_{i\le k}c_i$ and noting that $a_{k+1}/a_k\to c>1$), we get $B_\ast:=\sup_k B_k<\infty$, and hence $\sum_k\sigma^2_k(w)/k^2<\infty$ almost surely.

Given $w$, the $\tau_{k+1}$ are independent with mean $\mu_k(w)$ and variance $\sigma^2_k(w)$. Since $\sum_k\sigma^2_k(w)/k^2<\infty$, then the generalized Kolmogorov law; see e.g.,\cite{robbins1985convergence}, \cite[Corollary 7.4.1]{resnick2013probability}, implies that $$(T_n-E_0(T_n\mid w))/n\to 0\qquad \mathbb{RW}^r \text{-a.s.} $$

Since $E_0(T_n\mid w)=\sum_{k=0}^{n-1}\mu_k(w)
=n+2\sum_{k=0}^{n-1}\eta_k(w)$, it suffices to prove
\begin{equation}\label{eqcesaro}
\bar\eta_n(w):=\frac1n\sum_{k=0}^{n-1}\eta_k(w)\longrightarrow\frac{\bar s}{1-\bar s}=:L_\eta
\qquad\mathbb{Q}\text{-a.s.}
\end{equation}
By Lemma \ref{limitmean1}, we have $\mathbb{E}_{\mathbb{Q}}(\mu_k(w))\to(1+\bar s)/(1-\bar s)$. Consequently $\mathbb{E}_{\mathbb{Q}}(\eta_k(w))\to L_\eta$ and, by the Ces\`aro lemma, $\mathbb{E}_{\mathbb{Q}}(\bar\eta_n(w))\to L_\eta$. Therefore, to obtain the limit in \eqref{eqcesaro}, it suffices to prove that $\bar\eta_n-\mathbb{E}_{\mathbb{Q}}(\bar\eta_n)\to0$ $\mathbb{Q}$-a.s.

Recall that $\eta_k=s_k(1+\eta_{k-1})$ with $\eta_0=0$. Iterating this recursion yields, for any $1\le j<k$,
$$
\eta_k=\sum_{m=j+1}^k\prod_{i=m}^k s_i+\Bigl(\prod_{i=j+1}^k s_i\Bigr)\eta_j=: \eta'_{j,k}+P_{j+1,k}\eta_j .
$$
Observe that $P_{j+1,k}$ and $\eta'_{j,k}$ are functions only of $s_{j+1},\dots,s_k$ and hence are independent of $\eta_j$. It follows that
$$
\mathrm{Cov}_{\mathbb{Q}}(\eta_j,\eta_k)
=\mathbb{E}_{\mathbb{Q}}(P_{j+1,k})\,\mathrm{Var}_{\mathbb{Q}}(\eta_j)
=\Bigl(\prod_{i=j+1}^k\bar s_i\Bigr)\mathrm{Var}_{\mathbb{Q}}(\eta_j).
$$
Using \eqref{eqstar}, $\mathrm{Var}_{\mathbb{Q}}(\eta_j)\le C_0^2$, and for $j\ge i_0$ we also have $\prod_{i=j+1}^k\bar s_i\le \rho_0^{\,k-j}$ with some $\rho_0<1$. Hence, for $j\ge i_0$,
$$
\mathrm{Cov}_{\mathbb{Q}}(\eta_j,\eta_k)\le C_0^2\rho_0^{\,k-j}.
$$
Therefore, $\mathrm{Var}_{\mathbb{Q}}(\sum_{k=1}^n\eta_k)=O(n)$, since the covariances coming from the finitely many indices with $j<i_0$ contribute only $O(n)$. Consequently $\mathrm{Var}_{\mathbb{Q}}(\bar\eta_n)=O(1/n)$, so $\bar\eta_n\to L_\eta$ in $L^2(\mathbb{Q})$.

Along the subsequence $n_m=m^2$, these variances are summable, and thus $\bar\eta_{m^2}\to L_\eta$ almost surely. Finally, since $\eta_k\ge0$, for $m^2\le n<(m+1)^2$ we have
$$
\frac{S_{m^2}}{(m+1)^2}\le\bar\eta_n\le\frac{S_{(m+1)^2}}{m^2},
$$
where $S_n=\sum_{k=1}^n\eta_k$ and both bounds converge to $L_\eta$. This establishes \eqref{eqcesaro}.
Combining them all, we get $T_n/n\to(1+\bar s)/(1-\bar s)$ $\mathbb{S}$-a.s.
\end{proof}

The following lemma is used to find the speed through a limiting relationship between time and space; see \cite{zeitouni2009random,peterson2013lecture}. We will give the proof here for completeness.
\begin{lemma}\label{newlemma3}
The speed of a transient RWRE $\{X_n\}$ on $\mathbb{Z}^+$ is 
\begin{equation}
\lim_n \dfrac{X_n}{n}=\lim_n \dfrac{n}{T_n} \;\; a.s.
\end{equation}
\end{lemma}
\begin{proof}
Since $\{X_n\}$ is transient, we have $X_n\to \infty$ $\mathbb{S}$-a.s. Define the range of the RWRE to be $R_n=\max_{1\leq k\leq n}X_k$, which also tends to infinity. It is clear that $T_{R_n}\leq n<T_{R_n+1}$. Hence,
$$
\frac{T_{R_n}}{R_n}\leq \frac{n}{R_n}<\frac{T_{R_n+1}}{R_n},
$$
and therefore $\lim_n \frac{R_n}{n}=\lim_n \frac{n}{T_n}$ $\mathbb{S}$-a.s. This conclusion remains valid even in the case $\lim_n \frac{T_n}{n}=\infty$, where the same inequality yields $\lim_n \frac{R_n}{n}=0$ $\mathbb{S}$-a.s.

Next we show that $R_n$ and $X_n$ share the same linear speed. Since $X_n\leq R_n$, we obtain
$$
0\leq \limsup_n \frac{X_n}{n}\leq \limsup_n \frac{R_n}{n}.
$$
Thus, if the right-hand side is $0$, then the speed is $0$ as well. In what follows we exclude the case $\lim_n \frac{T_n}{n}=\infty$, because we have just seen that then $\lim_n \frac{X_n}{n}=0$.

On the other hand, note that $R_n-X_n\leq n-T_{R_n}$, and consequently
$$
\frac{R_n}{n}-\frac{X_n}{n}\leq 1-\frac{T_{R_n}}{R_n}\frac{n}{R_n}.
$$
It follows that $\limsup_n\Big(\frac{R_n}{n}-\frac{X_n}{n}\Big)\leq 0$, and hence $\limsup_n \frac{R_n}{n}\leq \liminf_n \frac{X_n}{n}$. Therefore,
$$
\lim_n \frac{X_n}{n}=\lim_n \frac{R_n}{n}=\lim_n \frac{n}{T_n}\qquad \mathbb{S}\text{-a.s.}
$$
\end{proof}

The following is the proof of Theorem \ref{theoremZ2} for the speed on $\textbf{RSST}w$.

\begin{proof}[Proof of Theorem \ref{theoremZ2}] 
Using the reduction outlined above, the process $\{|X_n|\}$ may be regarded as an RWRE on $\mathbb{Z}^+$ as in Theorem \ref{theoremZ1}, but with inverse-ratios $s_k=1/(d_k^+r_{k+1})$ in place of $1/r_k$. One verifies that the proof of Lemma \ref{newlemma1} remains valid under this modification, provided that $\bar s$ is replaced by $s_{Tw}$.

We begin by establishing \emph{transience.} Under Condition B$'$, we have $s_k\leq 1/\tilde c_k$ a.s., with $\tilde c_k\to \tilde c>1$; this serves the same purpose as the estimate $s_k\le 1/c_k$ in the original argument. Moreover,
$$
\mathbb{E}(\log d_k^+r_{k+1})\ge \log \tilde c_k \to \log \tilde c>0,
$$
and hence the effective drift tends to $+\infty$, implying that the walk is transient. Consequently, Lemma \ref{newlemma3} applies.

Let $\mu_k=1+2\sum_{j\le k}\prod_{i=j}^k s_i$, where $s_i=1/(d_i^+r_{i+1})$. Each summand in $\sigma_k^2$ is monotone nondecreasing in the collection $\{s_i\}$, and because $s_i\le 1/\tilde c_i$ we get the estimate $\sigma_k^2\le B_k$, with $B_k$ the deterministic sequence obtained by replacing every $s_i$ by $1/\tilde c_i$. As in the previous argument, $\sup_k B_k<\infty$, and therefore $\sum_k \sigma_k^2/k^2<\infty$ almost surely. Given the environment, the random variables $\tau_{k+1}$ are independent with mean $\mu_k$, so the generalized Kolmogorov law implies $(T_n-E_0(T_n\mid d^+,r))/n\to0$ almost surely.

Define $\eta_k=\sum_{j\le k}\prod_{i=j}^k s_i$. Then $\eta_k=s_k(1+\eta_{k-1})$, and Condition B$'$ together with \eqref{eqstar} implies the almost sure bound $\eta_k\le C_0$. As in Lemma \ref{newlemma1}, decompose $\eta_k$ into
$$
\eta_k=P_{j+1,k}\eta_j+\eta'_{j,k},
$$
where $P_{j+1,k}=\prod_{i=j+1}^k s_i$, and both $P_{j+1,k}$ and $\eta'_{j,k}$ are independent of $\eta_j$. It follows that
$$
\mathrm{Cov}(\eta_j,\eta_k)=\Bigl(\prod_{i=j+1}^k\bar s_i\Bigr)\mathrm{Var}(\eta_j).
$$
Furthermore, for $j$ large enough one has $\prod_{i=j+1}^k\bar s_i\le\tilde\rho_0^{\,k-j}$ for some $\tilde\rho_0<1$, and hence the variance of the partial sums grows like $O(n)$. Therefore $\bar\eta_n\to L_\eta$ in $L^2$, and by the squares-subsequence argument also almost surely. Applying Lemma \ref{limitmean1} to the effective environment yields $\mathbb{E}(\mu_k)\to(1+s_{Tw})/(1-s_{Tw})$, so $L_\eta=s_{Tw}/(1-s_{Tw})$ and consequently
$$
E_0(T_n\mid d^+,r)/n\to(1+s_{Tw})/(1-s_{Tw}) \quad \text{a.s.}
$$
Putting it all together gives $T_n/n\to(1+s_{Tw})/(1-s_{Tw})$ $\mathbb{S}$-a.s. Hence, Lemma \ref{newlemma3} implies
$$
X_n/n\to n/T_n\to(1-s_{Tw})/(1+s_{Tw})=v_{Tw}\quad \mathbb{S}\text{-a.s.}
$$
If $s_{Tw}\ge1$, the effective environment satisfies $\bar s\ge1$, and thus $T_n/n\to\infty$ and $v_{Tw}=0$ by Lemma \ref{newlemma3}. 
\end{proof}

\section{Conclusion}
Random walks on randomly evolving graphs of polynomial size in a random environment, such as the line and certain random trees, have been studied less than the case of exponentially growing trees. In this work, we examined the type problem for a random walk on a random tree in a non-homogeneous random environment. The environment is assumed to evolve randomly and independently of the trees. The logarithmic growth rates of the tree and the environment yield a P\'olya-type threshold for the model. We also analyzed the speed of the walk on the line and on random spherically symmetric trees with a random environment. Our findings extend earlier results for the same setting with a cascade-like environment. Two natural problems remain unresolved: the behavior at criticality, $\gamma+\alpha=1$ and $\gamma_v+\alpha=1$, where the resistance series diverges only logarithmically; and the speed on $\textbf{BPVET}w$, for which we provide an upper bound and state conjecture \ref{conj_2}.

\section*{Declaration of generative AI and AI-assisted technologies in the manuscript preparation process}
During the preparation of this work, the author(s) used Claude and Grammarly to edit and correct the English language grammar and flow. After using these tools, the author(s) reviewed and edited the content as needed and take full responsibility for the content of the published article.

\section*{Appendix}

Note that, for $\beta>-1$
\begin{equation}
\lim_n \frac{\sum_{k=1}^n k^\beta}{n^{\beta+1}}=\frac{1}{\beta+1}
\end{equation}
Also, for $\beta<-1$,
\begin{equation}\label{eq0}
\lim_n \frac{\sum_{k=n}^\infty
k^\beta}{n^{\beta+1}}=\frac{-1}{\beta+1}>0
\end{equation}
see \cite[p.41]{knopp}, and
\begin{equation}
\lim_n \frac{\sum_{k=1}^n \frac{1}{k}}{\log n}=1
\end{equation}

\begin{lemma}\label{limitlaws2}
Consider the two sequences $X_n,Y_n$ of nonnegative random
variables. If $\sum_n Y_n$ is divergent $a.s.$, then $\lim_n
\frac{X_n}{Y_n}=L$ $a.s.$ implies
\begin{equation}
\lim_n \frac{\sum_{k=1}^n X_k}{\sum_{k=1}^n Y_k}=L\quad a.s.
\end{equation}
\end{lemma}
\begin{proof}
Let $t_n=\frac{X_n}{Y_n}$. Then, as $n \longrightarrow \infty$,
$t_n \longrightarrow L$ on an event $A$ such that $P(A^c)=0$.
Also, $\lim_n \sum_{k=1}^n Y_k=\infty$ on an event $B$ such that
$P(B^c)=0$. Note that $P(A\cap B)^c=0$. Pick $\omega \in A\cap B$,
$\epsilon>0$ then there exists $n_0(\omega)$ sufficiently large
such that $k>n_0(\omega)$, yields
\begin{equation*}
L-\epsilon \leq t_k(\omega)\leq L+\epsilon
\end{equation*}
For $n>n_0$,
\begin{equation*}
\sum_{k=n_0+1}^n (L-\epsilon) Y_k(\omega)\leq \sum_{k=n_0+1}^n
X_k(\omega)\leq \sum_{k=n_0+1}^n (L+\epsilon)Y_k(\omega)
\end{equation*}
and so 
\begin{equation*} (\sum_{k=1}^n
Y_k(\omega)-\sum_{k=1}^{n_0} Y_k(\omega)) (L-\epsilon)\leq
\sum_{k=1}^n X_k(\omega)-\sum_{k=1}^{n_0} X_k(\omega)\leq
(\sum_{k=1}^n Y_k(\omega)-\sum_{k=1}^{n_0} Y_k(\omega))
(L+\epsilon)
\end{equation*}

\begin{equation*}
(1-\frac{\sum_{k=1}^{n_0} Y_k(\omega)}{\sum_{k=1}^n Y_k(\omega)})
(L-\epsilon)\leq \frac{\sum_{k=1}^n X_k(\omega)}{\sum_{k=1}^n
Y_k(\omega)}-\frac{\sum_{k=1}^{n_0} X_k(\omega)}{\sum_{k=1}^n
Y_k(\omega)}\leq (1-\frac{\sum_{k=1}^{n_0}
Y_k(\omega)}{\sum_{k=1}^n Y_k(\omega)}) (L+\epsilon)
\end{equation*} 
Hence, since $\lim_n \sum_{k=1}^n Y_k(\omega)=\infty$,
\begin{equation*}
(L-\epsilon)\leq \lim_n \frac{\sum_{k=1}^n
X_k(\omega)}{\sum_{k=1}^n Y_k(\omega)}\leq (L+\epsilon)
\end{equation*}
Since $\epsilon$ is arbitrarily small,
\begin{equation*}
\lim_n \frac{\sum_{k=1}^n X_k(\omega)}{\sum_{k=1}^n
Y_k(\omega)}=L.
\end{equation*}
Whence,
\begin{equation*}
\lim_n \frac{\sum_{k=1}^n X_k}{\sum_{k=1}^n Y_k}=L\quad
\text{on}\;A\cap B.
\end{equation*}
which was to be proven.
\end{proof}

We must note that the converse is not always true, unless
$\lim_n \frac{X_n}{Y_n}$ does exist. Note that in the following lemma, $L$ cannot be less than 1, but it can be equal to 1. See also Lemma \ref{limitlaws5} below for the case when $L=1$.

\begin{lemma}\label{limitlaws21}
Consider the two sequences $X_n,Y_n$ of nonnegative random
variables. If $\sum_n Y_n$ is converges $a.s.$, then $\lim_n
\frac{X_n}{Y_n}=L$ $a.s.$ implies
\begin{equation}
\lim_n \frac{\sum_{k=n}^\infty X_k}{\sum_{k=n}^\infty Y_k}=L\quad a.s.
\end{equation}
\end{lemma}
\begin{proof}
Let $t_n=\frac{X_n}{Y_n}$. Then, as $n \longrightarrow \infty$,
$t_n \longrightarrow L$ on an event $A$ such that $P(A^c)=0$.
Also, $\lim_n \sum_{k=n}^\infty Y_k=0$ on an event $B$ such that
$P(B^c)=0$. Note that $P(A\cap B)^c=0$. Pick $\omega \in A\cap B$,
$\epsilon>0$ then there exists $n_0(\omega)$ sufficiently large
such that $k>n_0(\omega)$, yields
\begin{equation*}
L-\epsilon \leq t_k(\omega)\leq L+\epsilon
\end{equation*}
For $n>n_0$,
\begin{equation*}
\sum_{k=n}^\infty (L-\epsilon) Y_k(\omega)\leq \sum_{k=n}^\infty
X_k(\omega)\leq \sum_{k=n}^\infty (L+\epsilon)Y_k(\omega)
\end{equation*}
Hence, 
\begin{equation*}
(L-\epsilon)\leq \lim_n \frac{\sum_{k=n}^\infty
X_k(\omega)}{\sum_{k=n}^\infty Y_k(\omega)}\leq (L+\epsilon)
\end{equation*}
Since $\epsilon$ is arbitrarily small,
\begin{equation*}
\lim_n \frac{\sum_{k=n}^\infty X_k(\omega)}{\sum_{k=n}^\infty
Y_k(\omega)}=L.
\end{equation*}
Whence,
\begin{equation*}
\lim_n \frac{\sum_{k=n}^\infty X_k}{\sum_{k=n}^\infty Y_k}=L\quad
\text{on}\;A\cap B.
\end{equation*}
which was to be proven.
\end{proof}

\begin{lemma}\label{limitlaw2x}
For any sequence $\{a_n:n\geq 1\}$ such that $\lim_n \frac{a_{n+1}}{a_n}=L$. If $L \geq 1$, then
\begin{enumerate}
\item[(a)] $\lim_n \frac{\sum_{k=1}^{n+1} a_k}{\sum_{k=1}^n a_k}=L$.
\item[(b)] $\lim_n \frac{a_{n+1}}{\sum_{k=1}^n a_k}=L-1$.
\item[(c)] $\lim_n \frac{a_{n}}{\sum_{k=1}^n a_k}=\dfrac{L-1}{L}$.
\end{enumerate}
If $L < 1$, then
\begin{enumerate}
\item[(d)] $\lim_n \frac{\sum_{k=n+1}^{\infty} a_k}{\sum_{k=n}^{\infty} a_k}=L$.
\item[(e)] $\lim_n \frac{a_{n}}{\sum_{k=n}^{\infty} a_k}=1-L$.
\item[(f)] $\lim_n \frac{a_{n+1}}{\sum_{k=n}^{\infty} a_k}=\dfrac{1-L}{L}$.
\end{enumerate}
\end{lemma}
\begin{proof}
Part (a) follows from Lemma \ref{limitlaws2}. Parts (b) and (c) follow from part (a) directly. Part (d) follows from Lemma \ref{limitlaws21}. Parts (e) and (f) follow from part (d) directly.

\end{proof}

\begin{lemma}\label{limitlawsb3}
If $a_n$ is a sequence of real numbers such that $a_n \geq 0$ for
all $n$, then
\begin{enumerate}
\item[(a)] $\lim_n n\,a_n=\lim_n \frac{\sum_{k=1}^n a_k}{\log n}$.
\item[(b)] $\lim_n n\,a_n=\lim_n \frac{\sum_{k=1}^n \log(1+a_k)}{\log n}$ if $a_n\geq
a_{n+1}$ for all $n$.
\end{enumerate}
Provided that $\lim_n n\,a_n$ exists.
\end{lemma}
\begin{proof}
\begin{enumerate}
\item From lemma \ref{limitlaws2},
\begin{equation*}
\begin{array}{l c l} \lim_n n\,a_n&=&\lim_n
\frac{a_n}{\frac{1}{n}}\\
\\
   &=&\lim_n \frac{\sum_{k=1}^n a_k}{\sum_{k=1}^n \frac{1}{k}}\\
\\
   &=&\lim_n \frac{\sum_{k=1}^n a_k}{\log n}\,
   \frac{\log n}{\sum_{k=1}^n \frac{1}{k}}\\
\\
   &=&\lim_n \frac{\sum_{k=1}^n a_k}{\log n}
\end{array}
\end{equation*}

\item $\lim_n (1+a_n)^n=e^{\lim_n n\,a_n}$ and so $\lim_n \log(1+a_n)^n=\lim_n n\,a_n$.
\\
Applying part (1) to $b_n=\log (1+a_n)$
\begin{equation*}
\lim_n \frac{\sum_{k=1}^n b_k}{\log n}=\lim_n n\,b_n=\lim_n
n\,a_n.
\end{equation*}
\end{enumerate}
\end{proof}

Define the partial sum of a
sequence $X_k$ to be $\widehat{S}_n=\sum_{k=1}^n X_k$. Also,
define the remainder to be $\widehat{R}_n=\sum_{k=n}^\infty X_k$.

\begin{lemma}\label{limitlaws3}
Consider a sequence $X_n$ of nonnegative random variables.
\newline 1) If $\lim_n \frac{\log X_n}{\log n}=L>-1$ $a.s.$, then $\lim_n
\widehat{S}_n$ is almost surely divergent and
\begin{equation}
\lim_n \frac{\log \widehat{S}_n}{\log n}=L+1\quad a.s.
\end{equation}
2) If $\lim_n \frac{\log X_n}{\log n}=L<-1$ $a.s.$, then $\lim_n
\widehat{S}_n$ is almost surely convergent and
\begin{equation}
\lim_n \frac{\log \widehat{R}_n}{\log n}=L+1\quad a.s.
\end{equation}
3) If $\lim_n \frac{\log X_n}{\log n}=-1$ $a.s.$, then
\begin{equation}
\lim_n \frac{\log \widehat{S}_n}{\log n}=0\quad a.s.
\end{equation}
\end{lemma}
\begin{proof}
Let $t_n=\frac{\log X_n}{\log n}$. Then, as $n \longrightarrow
\infty$, $t_n \longrightarrow L$ on an event $A$ such that
$P(A^c)=0$. Therefore, $X_n=n^{t_n}$ and so
$\widehat{S}_n=\sum_{k=1}^n X_k=\sum_{k=1}^n k^{t_k}$. Let $\omega
\in A$ and $\epsilon>0$ be given: then there exists $n_0(\omega)$
sufficiently large such that for $k>n_0(\omega)$,
\begin{equation*}
L-\epsilon \leq t_k(\omega)\leq L+\epsilon
\end{equation*}
For $n>n_0$,
\begin{equation*}
\sum_{k=n_0+1}^n k^{L-\epsilon}\leq \sum_{k=n_0+1}^n X_k(\omega)
\leq \sum_{k=n_0+1}^n k^{L+\epsilon}
\end{equation*}
Therefore,
\begin{equation}\label{eq1}
\sum_{k=n_0+1}^n k^{L-\epsilon}\leq
\widehat{S}_n(\omega)-\widehat{S}_{n_0}(\omega) \leq
\sum_{k=n_0+1}^n k^{L+\epsilon}
\end{equation}
\newline 1) If $L>-1$ then for arbitrary small $\epsilon$ the
limit of the left-hand side of the inequality diverges and so is
$\lim_n \widehat{S}_n(\omega)$. Therefore, $\lim_n \widehat{S}_n$
is divergent on $A$.\newline 2) If $L<-1$ then for arbitrary small
$\epsilon$, the limit of the right-hand side of the inequality
converges and so is $\lim_n \widehat{S}_n(\omega)$. Therefore,
$\lim_n \widehat{S}_n$ is convergent on $A$.

Now, in case that $L>-1$, from \eqref{eq1}, 
\begin{equation}\label{eq2b}
 \frac{\log (\sum_{k=1}^n
k^{L-\epsilon}-\sum_{k=1}^{n_0} k^{L-\epsilon})}{\log n}\leq
\frac{\log (\widehat{S}_n(\omega)-\widehat{S}_{n_0}(\omega))}{\log
n}\leq \frac{\log (\sum_{k=1}^n k^{L+\epsilon}-\sum_{k=1}^{n_0}
k^{L+\epsilon})}{\log n}
\end{equation}

The limiting value of the left-hand side of the inequality would
be
\begin{equation*}
\lim_n \frac{\log \frac{(\sum_{k=1}^n
k^{L-\epsilon}-\sum_{k=1}^{n_0}
k^{L-\epsilon})}{n^{L+1-\epsilon}}+\log (n^{L+1-\epsilon})}{\log
n}=L+1-\epsilon
\end{equation*}
The limiting value of the right-hand side of the inequality would
be
\begin{equation*}
\lim_n \frac{\log \frac{(\sum_{k=1}^n
k^{L+\epsilon}-\sum_{k=1}^{n_0}
k^{L+\epsilon})}{n^{L+1+\epsilon}}+\log (n^{L+1+\epsilon})}{\log
n}=L+1+\epsilon
\end{equation*}
and since in case that $L>-1$, $\lim_n \widehat{S}_n$ is divergent
$a.s.$ then for $\omega\in A$,
\begin{equation*}
\lim_n \frac{\log
(\widehat{S}_n(\omega)-\widehat{S}_{n_0}(\omega))}{\log n}=\lim_n
\frac{\log
(1-\frac{\widehat{S}_{n_0}(\omega)}{\widehat{S}_n(\omega)})+\log
\widehat{S}_n(\omega)}{\log n}=\lim_n \frac{\log
\widehat{S}_n(\omega)}{\log n}
\end{equation*}
So
\begin{equation*}
(L+1)-\epsilon\leq \lim_n \frac{\log \widehat{S}_n(\omega)}{\log
n}\leq (L+1)+\epsilon
\end{equation*}
Hence,
\begin{equation*}
\lim_n \frac{\log \widehat{S}_n(\omega)}{\log n}=L+1.
\end{equation*}
Whence,
\begin{equation*}
\lim_n \frac{\log \widehat{S}_n}{\log n}=L+1\quad \text{on}\;A.
\end{equation*}
In case that $L<-1$, pick $n>n_0$ and hence for $\omega \in A$
\begin{equation*}
\sum_{k=n}^\infty k^{L-\epsilon}\leq \sum_{k=n}^\infty X_k(\omega)
\leq \sum_{k=n}^\infty k^{L+\epsilon}
\end{equation*}
Using equation \eqref{eq0}
\begin{equation*}
(L+1)-\epsilon\leq \lim_n \frac{\log \widehat{R}_n(\omega)}{\log
n}\leq (L+1)+\epsilon
\end{equation*}
So,
\begin{equation*}
\lim_n \frac{\log \widehat{R}_n}{\log n}=L+1 \quad a.s.
\end{equation*}
In case that $L=-1$, it follows from inequality \eqref{eq1} that
\begin{equation*}
-\epsilon\leq \lim_n \frac{\log \widehat{S}_n(\omega)}{\log n}\leq
\epsilon
\end{equation*}
and the result follows.
\end{proof}

\begin{lemma}\label{limitlaws4}
Consider a sequence $X_n$ of nonnegative random variables.
\newline 1) If $\lim_n \frac{\log X_n}{\log n}=\infty$ $a.s.$, then
$\lim_n \widehat{S}_n$ is almost surely divergent and
\begin{equation*}
\lim_n \frac{\log \widehat{S}_n}{\log n}=\infty \quad a.s.
\end{equation*}
\newline 2) If $\lim_n \frac{\log X_n}{\log n}=-\infty$ $a.s.$, then
$\lim_n \widehat{S}_n$ is almost surely convergent and
\begin{equation*}
\lim_n \frac{\log \widehat{R}_n}{\log n}=-\infty \quad a.s.
\end{equation*}
\end{lemma}

\begin{proof}
1) Let $t_n=\frac{\log X_n}{\log n}$ such that $t_n\rightarrow
\infty$ as $n\rightarrow \infty$ on an event $A$ and $P(A^c)=0$.
Hence, let $\omega \in A$ and arbitrary large $M> 0$ be given,
then there exists $n_0(\omega)$ sufficiently large such that for
$k>n_0(\omega)$,
\begin{equation*}
t_k(\omega)>M
\end{equation*}
and so
\begin{equation*}
X_k(\omega)>k^M
\end{equation*}
and for $n>n_0$
\begin{equation*}
\widehat{S}_n(\omega)-\widehat{S}_{n_0}(\omega)>\sum_{k=n_0+1}^{n}
k^M
\end{equation*}
and the results follow. \newline 2) It is quite easy using the same argument.
\end{proof}

The following lemma gives an alternative form of Raabe's test.

\begin{lemma}\label{limitlaws5}
For any nonnegative sequence $\{a_n:n\geq 1\}$ such that $\lim_n \dfrac{a_{n+1}}{a_n}=1$ and $\lim_n\dfrac{\log a_n}{\log n}=L$,
\begin{enumerate}
\item[(a)] $\lim_n n(\dfrac{a_{n+1}}{a_n}-1)=L$.
\item[(b)] If $L \geq -1$, $\lim_n \dfrac{n a_{n}}{\sum_{k=1}^{n} a_k}=1+L$.
\item[(c)] If $L<-1$, then $\lim_n \dfrac{n a_{n}}{\sum_{k=n}^{\infty} a_k}=-1-L$.
\end{enumerate}
\end{lemma}
\begin{proof}
For part (a), $$\dfrac{\log a_n}{\log n}=\dfrac{\log(a_1)+\log \prod_{i=1}^{n-1} \frac{a_{i+1}}{a_{i}}}{\log n}=\dfrac{\log(a_1)+\sum_{i=1}^{n-1} \log  \frac{a_{i+1}}{a_{i}}}{\log n}.$$ Hence, $$\lim_n \dfrac{\log a_n}{\log n}=\lim_n \dfrac{\sum_{i=1}^{n-1} \log  \frac{a_{i+1}}{a_{i}}}{\log n}=\lim_n n \log  \frac{a_{n+1}}{a_{n}}$$ by lemma \ref{limitlawsb3}.
On the other hand, $n \log  \frac{a_{n+1}}{a_{n}}=n \log  (\frac{a_{n+1}}{a_{n}}-1+1)=\log [ (1+(\frac{a_{n+1}}{a_{n}}-1))^n ]$. But since $\lim_n (1+(\frac{a_{n+1}}{a_{n}}-1))^n =\exp(\lim_n n (\frac{a_{n+1}}{a_{n}}-1)) $ then the result follows.

For part (b), $$L+1 =\lim_n \dfrac{\log \sum_{k=1}^{n} a_k}{\log n} =
\lim_n n(\dfrac{\sum_{k=1}^{n+1} a_k}{\sum_{k=1}^{n} a_k}-1)=\lim_n \dfrac{n a_{n+1}}{\sum_{k=1}^{n} a_k}.$$

For part (c), $$L+1 =\lim_n \dfrac{\log \sum_{k=n}^{\infty} a_k}{\log n} =
\lim_n n(\dfrac{\sum_{k=n+1}^{\infty} a_k}{\sum_{k=n}^{\infty} a_k}-1)=\lim_n \dfrac{-n a_{n}}{\sum_{k=n}^{\infty} a_k}.$$ On the other hand, $L<-1$ implies that $\sum_n a_n$ is convergent, see lemma \ref{limitlaws3}, then
$$0 =\lim_n \dfrac{\log \sum_{k=1}^{n} a_k}{\log n} =
\lim_n n(\dfrac{\sum_{k=1}^{n+1} a_k}{\sum_{k=1}^{n} a_k}-1)=\lim_n \dfrac{n a_{n+1}}{\sum_{k=1}^{n} a_k}.$$ 
\end{proof}

\begin{remark}
If $L > -1$, Lemma \ref{limitlaws5}(b) implies that $\lim_n \dfrac{n^{1+c} a_{n}}{\sum_{k=1}^{n} k^{c}\,a_k}=1+c+L$ for $c \geq 0$. Similarly, if $L < -1$,  Lemma \ref{limitlaws5}(c) implies that $\lim_n \dfrac{n^{1-c} a_{n}}{\sum_{k=n}^{\infty} k^{-c}\,a_k}=-1-c-L$ for $c \geq 0$. Moreover, $\lim_n \dfrac{n^{1+c} a_{n}}{\sum_{k=1}^{n} k^{c}\,a_k}=1+c+L$ for $c+L > -1$. 
\end{remark} 

\begin{remark}
If $a_n$ is a regularly varying sequence of exponent $L \geq 0$ then its partial sum is a regularly varying sequence of exponent $L+1$.  
\end{remark}

\bibliography{ref,previous}

\end{document}